\documentclass[12pt]{amsart}

\usepackage{verbatim}
\usepackage{amssymb}
\usepackage[pdftex]{graphicx}
\usepackage[all]{xy}
\usepackage{color}
\usepackage[colorlinks=true,linkcolor=blue,citecolor=red]{hyperref}

\newtheorem{theorem}{Theorem}[section]
\newtheorem{corollary}[theorem]{Corollary}
\newtheorem{lemma}[theorem]{Lemma}
\newtheorem{proposition}[theorem]{Proposition}

\theoremstyle{definition} \newtheorem{definition}[theorem]{Definition}
\newtheorem{remark}[theorem]{Remark}
\newtheorem{example}[theorem]{Example}

\newcommand{\K}{\Bbbk}

\DeclareMathOperator{\trop}{trop}
\DeclareMathOperator{\inn}{in}

\DeclareMathOperator{\im}{im}
\DeclareMathOperator{\Hom}{Hom}

\DeclareMathOperator{\conv}{conv}
\DeclareMathOperator{\val}{val}

\DeclareMathOperator{\spann}{span}
\DeclareMathOperator{\Gr}{Gr}

\newcommand{\PuiseuxC}{\ensuremath{\mathbb C \{\!\{t \} \! \}}}

\begin{document}

\title{Polyhedral structures on tropical varieties}

\author{Diane Maclagan}

\address{Mathematics Institute\\
University of Warwick\\
Coventry, CV4 7AL\\
United Kingdom
}

\email{D.Maclagan@warwick.ac.uk}

\maketitle

\section{Introduction}

One reason for the recent success of tropical geometry is that
tropical varieties are easier to understand than classical varieties.
This is largely because they are discrete, combinatorial objects
having the structure of a polyhedral complex.  The purpose of these
expository notes is to give the Gr\"obner perspective on the origin of
this polyhedral complex structure.  We review the basic definitions of
tropical geometry in the rest of this section, before stating the main
theorems in the next section.  The last section is devoted to the
proofs of these theorems, some of which are new.

We begin by setting notation.  Throughout the paper we work with a fixed field
$K$ with a nontrivial valuation $\val : K^* \rightarrow \mathbb R$.
We denote by $R$ the valuation ring of $K$: $R = \{a \in K : \val(a)
\geq 0\}$.  The ring $R$ is a local ring with maximal ideal
$\mathfrak{m} = \{ a \in K : \val(a) >0 \}$ and residue field $\K =
R/\mathfrak{m}$.  For $a \in R$ we denote by $\overline{a}$ the image
of $a$ in $\K$.  We denote by $\Gamma \subseteq \mathbb R$ the image
of the valuation $\val$.  If $\Gamma \neq \{0\}$ then we assume $1 \in
\Gamma$ as this can be guaranteed by replacing $\val$ by a positive
multiple.

 We do not assume that $K$ is complete, but will sometimes require
 that it be algebraically closed.  Given an ideal over a field $K$
 without  a nontrivial valuation (for example, $K=\mathbb C$), we can extend scalars to work over the field of generalized power series with coefficients in $K$.

\begin{definition} \label{d:tropicalvariety}
  For $f = \sum_{u \in \mathbb
  Z^n} c_u x^u \in K[x_1^{\pm 1},\dots,x_n^{\pm 1}]$ the set
$\trop(V(f))$ is the non-linear locus of the piecewise linear function
$\trop(f)$ given by $\trop(f)(w) = \min(\val(c_u)+ w \cdot u)$.  Let $X \subseteq T^n \cong (K^*)^n$. The
tropical variety of $X$ is
$$\trop(X) = \bigcap_{f \in I(X)} \trop(V(f)),$$ where $I(X) \subseteq 
K[x_1^{\pm 1},\dots,x_n^{\pm 1}]$ is the ideal of $X$.
\end{definition}

The fundamental theorem of tropical algebraic geometry is the following:

\begin{theorem} \label{t:fundamentaltheorem}
For a variety $X \subseteq T^n \cong (K^*)^n$, where $K=\overline{K}$,
the set $\trop(X)$ equals the closure in the Euclidean topology on
$\mathbb R^n$ of the set $$\val(X) = \{ (\val(x_1),\dots, \val(x_n) ) :
x = (x_1,\dots,x_n) \in X \}.$$
\end{theorem}

See, for example, \cite[Section 3.2]{TropicalBook} for a proof.
Theorem~\ref{t:fundamentaltheorem} gives a second interpretation of
the tropical variety $\trop(X)$ as a ``combinatorial
shadow'' of the variety $X$.  We now describe a third way to
understand it, which uses the theory of Gr\"obner
bases.

We now assume that there exists a splitting of the valuation.  This is
a group homomorphism $\Gamma \rightarrow K^*$ sending $w \in \Gamma$
to $t^w \in K^*$ with $\val(t^w)=w$.  If $K$ is the field of Puiseux
series $\PuiseuxC$ with coefficients in $\mathbb C$, we may take the
splitting that sends $w \in \mathbb Q$ to $t^w \in \PuiseuxC$.  If
$K=\mathbb Q_p$, we may take the splitting that sends $w \in \mathbb Z$
to $p^w$.  If $K$ is algebraically closed, then such a splitting
always exists; see \cite[Lemma 2.1.13]{TropicalBook}.

\begin{definition}
Fix $w \in \Gamma^n$.  For a polynomial $f = \sum_{u \in \mathbb Z^n}
c_u x^u \in K[x_1^{\pm 1},\dots,x_n^{\pm 1}]$, let $W = \trop(f)(w):=\min(\val(c_u)
+ w \cdot u)$.  We set 
\begin{align*} \inn_w(f) & = \overline{t^{-W}f(t^{w_1}x_1,\dots,t^{w_n}x_n) }\\
&  = \sum_{u \in
  \mathbb Z^n} \overline{t^{w \cdot u -W} c_u} x^u \\
 & = \sum_{\val(c_u)+w
  \cdot u = W} \overline{t^{-\val(c_u)}c_u} x^u \in \K[x_1^{\pm
    1},\dots,x_n^{\pm 1}].\\ \end{align*}
\end{definition}

\begin{example}
Let $f = 6x^2 +5xy+7y^2 \in \mathbb Q[x^{\pm 1},y^{\pm 1}]$, where
$\val$ is the $2$-adic valuation on $\mathbb Q$.  For $w = (1,2)$, we
have $W = \min(3,3,4)=3$, so
\begin{align*}
\inn_w(f) & = \overline{1/8(6(2x)^2+5(2x)(4y)+7(4y)^2)} \\
& = \overline{3x^2+5xy+14y^2} \\
& =x^2+xy \in \mathbb Z/2\mathbb Z[x^{\pm 1}, y^{\pm 1}].
\end{align*}
\end{example}

\begin{definition}
Let $I$ be an ideal in $K[x_1^{\pm 1},\dots,x_n^{\pm 1}]$.  The initial ideal of $I$ is 
$$\inn_w(I) = \langle \inn_w(f) : f \in I \rangle \subseteq
\K[x_1^{\pm 1},\dots,x_n^{\pm 1}].$$ A subset $\{g_1,\dots,g_r\}$ of
$I$ is a Gr\"obner basis for $I$ with respect to $w$ if $\inn_w(I) =
\langle \inn_w(g_1),\dots,\inn_w(g_r) \rangle$.
\end{definition}

This generalizes the notion of Gr\"obner bases for ideals in a
polynomial ring with no valuations considered.  An excellent
elementary reference for that case is \cite{CLO}.  As in that
situation, a generating set for $I$ need not be a Gr\"obner basis.

\begin{example}
Let $I=\langle x+2y, x+4z \rangle \subseteq \mathbb Q[x^{\pm 1},y^{\pm
    1},z^{\pm 1}]$, where $\mathbb Q$ has the $2$-adic valuation.  For
$w=(1,1,1)$, we have $\inn_w(I) = \langle x,y \rangle \subseteq
\mathbb Z/2\mathbb Z[x^{\pm 1},y^{\pm 1}, z^{\pm 1}]$, even though
$\inn_w(x+2y)=\inn_w(x+4z)=x$.
\end{example}

\begin{remark}
For $f \in K[x_1^{\pm 1},\dots,x_n^{\pm 1}]$, the non-linear locus of
the function $\trop(f)$ is the locus where the minimum is achieved at
least twice, and thus is the closure of the collection of $w$ for
which $\inn_w(f)$ is not a monomial.  This means that, if $\Gamma$ is
dense in $\mathbb R$, $\trop(X)$ is the closure of those $w \in
\Gamma^n$ for which $\inn_w(I(X)) \neq \langle 1 \rangle$.
\end{remark}

\section{Gr\"obner complex}

In this section we develop the theory of the Gr\"obner complex of an
ideal, which leads to a polyhedral structure on $\trop(X)$.  The
Gr\"obner complex generalizes the Gr\"obner fan \cite{BayerMorrison},
\cite{MoraRobbiano} from standard Gr\"obner theory.  It was first
described in the thesis of Speyer \cite{SpeyerThesis}.  In this
section we restrict to the case that $I$ is a homogeneous ideal in the
(non-Laurent) polynomial ring $K[x_0,\dots,x_n]$.  We assume that
$\Gamma=\im \val$ is a dense subset of $\mathbb R$ containing $\mathbb
Q$.  This follows from the assumption that $1 \in \Gamma$ if $K$ is
algebraically closed.  If $I$ is defined over a field with a trivial
valuation, choose $K$ to be any extension field with a nontrivial
valuation, and consider $I\otimes K$; the results do not depend on the
choice of $K$.  For $w \in \Gamma^{n+1}$, the initial form $\inn_w(f)$
of a polynomial $f \in K[x_0,\dots,x_n]$ is defined as in the Laurent
polynomial case: $\inn_w(f) =
\overline{t^{-\trop(f)(w)}f(t^{w_1}x_1,\dots,t^{w_n}x_n)}$.  The
initial ideal of an ideal is similarly the ideal generated by all
initial forms of polynomials in the ideal.

\begin{definition}
Fix $w \in \Gamma^{n+1}$.  Define
$$C_I[w]= \{w' \in \Gamma^{n+1} : \inn_{w'}(I)=\inn_w(I) \}.$$ We
denote by $\overline{C_I[w]}$ the closure of $C_I[w]$ in the usual
Euclidean topology on $\mathbb R^{n+1}$.
\end{definition}

\begin{example}
Let $f = 3x+8y+6z \in \mathbb Q[x,y,z]$, where $\mathbb Q$ has the
$3$-adic valuation, and let $I= \langle f \rangle$.  Fix $w =
(1,1,1)$.  Then $\trop(f)(w)=\min(2,1,2)=1$, so $\inn_w(f) =
\overline{1/3(9x+24y+18z)} = 2y \in \mathbb Z/3\mathbb Z[x,y,z]$.  Then 
\begin{align*}
C_I[w] & = \{w' \in \Gamma^3 : \inn_{w'}(I) = \langle y \rangle \}\\
& = \{ w' \in \Gamma^3 : w'_1+1  > w_2', w_3'+1 > w_2' \}.
\end{align*}
The closure $\overline{C_I[w]}$ is then $\{w' \in \mathbb R^3 : w_1'+1
\geq w_2, w_3'+1 \geq w_2'\}$.  To visualize this, we note that if
$w'$ lies in $\overline{C_I[w]}$, then so does $w'+\lambda(1,1,1)$ for
any $\lambda \in \mathbb R$, so we may quotient by the span of
$(1,1,1)$ to draw pictures.  The region $\overline{C_I[w]}$ is the
shaded region on the left of Figure~\ref{f:CI[w]}, where we have
chosen the representatives for cosets in $\mathbb R^3/\mathbb
R(1,1,1)$ with last coordinate zero.

\begin{figure}
\center{\resizebox{11cm}{!}{\input{GrobnerComplex.pdftex_t}}}
\caption{\label{f:CI[w]}}
\end{figure}

The picture on the right of Figure~\ref{f:CI[w]} shows the other possible initial ideals of $I$, and the corresponding regions $\overline{C_I[w]}$.
\end{example}

\begin{remark} \label{r:lineality}
Note that if $I$ is a homogeneous ideal in $K[x_0,\dots,x_n]$, then
$\inn_{w+\lambda \mathbf{1}}(I) = \inn_w(I)$ for any $\lambda \in
\mathbb R$, where $\mathbf{1} = (1,\dots, 1)$. 
\end{remark}

Recall that a polyhedral complex is a collection of polyhedra which
contains all faces of any polyhedron in the collection and for which
the intersection of any two polyhedra is either empty or a common
face.  The key result of this section, which is proved in the
following section, is that there are only finitely many of the sets
$\overline{C_I[w]}$ as $w$ varies over $\Gamma^{n+1}$ and these sets
are polyhedra that fit together to form a polyhedral complex.

Every polyhedron in $\mathbb R^{n+1}$ can be written in the form $P=
\{ x \in \mathbb R^{n+1} : A x \leq b \}$ where $A$ is an $s \times
(n+1)$ matrix and $b \in \mathbb R^s$.  We say that $P$ is
$\Gamma$-rational if the entries of $A$ are rational and $b \in
\Gamma^s$.  This means that all facet normals of $P$ are vectors in
$\mathbb Q^{n+1}$ and all vertices of $P$ are elements of $\Gamma^{n+1}$.  A
polyhedral complex $\Sigma$ is $\Gamma$-rational if all polyhedra in
$\Sigma$ are $\Gamma$-rational.

\begin{theorem} \label{t:Grobnercomplex}
Fix a homogeneous ideal $I \subseteq K[x_0,\dots,x_n]$.  Then 
$\{ \overline{C_I[w]} : w \in \Gamma^{n+1} \}$ forms a finite $\Gamma$-rational polyhedral complex.
\end{theorem}

The polyhedral complex of Theorem~\ref{t:Grobnercomplex} is called the
{\em Gr\"obner complex}.  In the case that the residue field $\K$ is a
subfield of $K$, and $I$ is defined over $\K$ (such as when $I
\subseteq \mathbb C[x_0,\dots,x_n]$, where it is standard to take $K =
\PuiseuxC$), the Gr\"obner complex is a rational polyhedral fan, which
is known as the Gr\"obner fan.  This is well studied in the usual
Gr\"obner literature; see \cite{MoraRobbiano} or \cite{BayerMorrison}
for the original works, or \cite[Chapter 2]{GBCP} or \cite[Chapter
  2]{IndiaNotes} for expositions.  The Gr\"obner complex appears in
Speyer's thesis~\cite{SpeyerThesis}, though our exposition is
different.

The lineality space of a polyhedral complex $\Sigma$ is the largest
subspace $L$ for which if $u \in \sigma$ for any $\sigma \in \Sigma$,
and $l \in L$, then $u +l \in \sigma$.  Remark~\ref{r:lineality} thus
says that $\mathbb R \mathbf{1}$ is in the lineality space of the
Gr\"obner complex, so we can draw it in $\mathbb R^{n+1}/\mathbb R
\mathbf{1}$. 

\begin{example}
Let $I = \langle y^2z-x^3-x^2z-p^4z^3 \rangle \subseteq \mathbb
Q[x,y,z]$, where $\mathbb Q$ has the $p$-adic valuation for some prime
$p$.  For $f=y^2z-x^3-x^2z-p^4z^3$, we have $\trop(f) =
\min(2y+z,3x,2x+z,3z+4)$.  The Gr\"obner complex is illustrated in
Figure~\ref{f:GrobnercomplexMattExample}.

\begin{figure}
\center{\resizebox{6cm}{!}{\input{MattEg.pdftex_t}}}
\caption{\label{f:GrobnercomplexMattExample}}
\end{figure}

\end{example}

The relevance of Theorem~\ref{t:Grobnercomplex} in the tropical
context is that it gives the structure of a polyhedral complex to
$\trop(X)$.  

Given an ideal $I \subset K[x_1^{\pm 1}, \dots,x_n^{\pm 1}]$, we
denote by $I^h \in K[x_0,\dots,x_n]$ the homogenization of $I \cap
K[x_1,\dots,x_n]$.  This is the ideal $I^h = \langle \tilde{f} : f\in
I \cap K[x_1,\dots,x_n] \rangle$, where $\tilde{f} =
x_0^{\deg(f)}f(x_1/x_0,\dots,x_n/x_0)$ is the homogenization of $f$.
 The support of a polyhedral complex $\Sigma \subseteq
\mathbb R^{n+1}$ is the collection of vectors $w \in \mathbb R^{n+1}$
with $w \in \sigma$ for some $\sigma \in \Sigma$.

\begin{corollary}
Let $X$ be a subvariety of $T^n$.  Then there is a finite $\Gamma$-rational
polyhedral complex $\Sigma$ whose support $|\Sigma|$ equals
$\trop(X)$.
\end{corollary}

\begin{proof}
Let $I=I(X) \subseteq K[x_1^{\pm 1},\dots,x_n^{\pm 1}]$ be the ideal
of polynomials vanishing on $X$, and let $I^h$ be its homogenization.
It is straightforward to check that for $w \in \Gamma^n$ we have
$\inn_{(0,w)}(I^h)|_{x_0=1} = \inn_w(I)$, where the equality is as
ideals in $\K[x_1^{\pm 1},\dots,x_n^{\pm 1}]$; see \cite[Proposition
  2.6.2]{TropicalBook} for details.  Thus $\inn_w(I) = \langle 1
\rangle$ if and only if $\inn_{(0,w)}(I^h) \subseteq
\K[x_0,\dots,x_n]$ contains a monomial.  Let $\Sigma$ be the subset of
the Gr\"obner complex defined by $\{ \overline{C_{I^h}[(0,w)]} :
\inn_{(0,w)}(I^h) \text{ does not contain a monomial} \}$.  This is a
subset of a $\Gamma$-rational polyhedral complex, so the slice $w_0=0$
is also a $\Gamma$-rational polyhedral complex.  Since the polyhedra
in $\Sigma$ intersect correctly, to show that $\Sigma \cap \{w_0=0\} =
\trop(X)$, it only remains to check that if $w' \in
\overline{C_{I^h}[(0,w)]} \setminus C_{I^h}[(0,w)]$, then
$\inn_{w'}(I^h)$ also contains no monomials.  This follows from
Corollary~\ref{l:initialofinitial} in the next section, as if $w' \in
\overline{C_{I^h}[(0,w)]}$ then there is $\mathbf{v} \in \Gamma^{n+1}$
for which $w'+\epsilon \mathbf{v} \in C_{I^h}[(0,w)]$ for all
$\epsilon$ sufficiently small.  Therefore $\inn_{(0,w)}(I^h) =
\inn_{\mathbf{v}}(\inn_{w'}(I^h))$ is an initial ideal of
$\inn_{w'}(I^h)$ by Corollary~\ref{l:initialofinitial}.  This means that  if
$\inn_{w'}(I^h)$ contains a monomial then so does $\inn_{(0,w)}(I^h)$.
Thus if $\overline{C_{I^h}[(0,w)]} \in \Sigma$, we also have
$\overline{C_{I^h}[w']} \in \Sigma$ as required.
\end{proof}

A drawback of the definition of a tropical variety given in
Definition~\ref{d:tropicalvariety} is that a priori it requires taking
the intersection over infinitely many tropical hypersurfaces
$\trop(V(f))$.  A second tropical consequence of
Theorem~\ref{t:Grobnercomplex} is that this infinite intersection is
in fact a finite intersection.

\begin{definition}
Let $I \subseteq K[x_1^{\pm 1},\dots,x_n^{\pm 1}]$ be an ideal.  A collection $\{ f_1,\dots,f_r \} \subseteq I$ is a {\em tropical basis} for $I$ if 
$$\trop(V(I)) = \bigcap_{i=1}^r \trop(V(f_i)),$$ and $I = \langle
f_1,\dots,f_r \rangle$.
\end{definition}

\begin{theorem} \label{t:tropicalbasis}
Let $I \subseteq K[x_1^{\pm 1},\dots,x_n^{\pm 1}]$ be an ideal.  Then
a tropical basis for $I$ always exists.
\end{theorem}

\begin{proof}
The Gr\"obner complex $\Sigma(I)$ of $I^h$ is a polyhedral complex in
$\,\mathbb{R}^{n+1}$ with lineality space containing $\mathbb R
\mathbf{1}$.  For each of the finitely many polyhedra $\sigma^{(i)}$
in that complex, we select one representative vector $w^{(i)} \in
\Gamma^{n+1}$.  For each index $i$ such that $\,\inn_{w^{(i)}}(I^h)$
contains a monomial we select a polynomial $f^{(i)} \in I^h$ such that
$\inn_{w^{(i)}}(f^{(i)})$ is a monomial $x^{u_i}$.  Choose
$\mathbf{v}_i$ with $\inn_{\mathbf{v}_i}(\inn_{w^{(i)}}(I^h))$ a
monomial ideal; this is possible by Lemma~\ref{l:monomial1}.  By
Corollary~\ref{l:initialofinitial} we can find $\epsilon>0$ such that
$\inn_{\mathbf{v}_i}(\inn_{w^{(i)}}(I^h)) = \inn_{w^{(i)}+\epsilon
  \mathbf{v}_i}(I^h)$.  By Lemma~\ref{l:Hilbertfunction} there is a
polynomial $g^{(i)} \in I$ of the form $x^{u_i} + \sum c_{ai}x^a$,
where $c_{ai} \neq 0$ implies that $x^a \not \in
\inn_{\mathbf{v}_i}(\inn_{w^{(i)}}(I^h))$.  Then for every $w \in
\Gamma^{n+1}$ with $\inn_w(I^h)=\inn_{w^{(i)}}(I^h)$ we claim that
$\inn_w(g^{(i)})= x^{u_i}$.  Indeed,
$\inn_{\mathbf{v}_i}(\inn_w(g^{(i)})) \in
\inn_{\mathbf{v}_i}(\inn_{w^{(i)}}(I^h))$, and every monomial occurring
in this polynomial must occur in $g^{(i)}$, but also be in the
monomial ideal $\inn_{\mathbf{v}_i}(\inn_{w^{(i)}}(I^h))$, so must be
$x^{u_i}$.  Thus $\inn_w(g^{(i)}) = x^{u_i}+ \sum b_a x^a$ where $x^a
\not \in \inn_{\mathbf{v}_i}(\inn_{w}(I^h))$.  Since $x^{u_i} \in
\inn_w(I^h)$, this means that $\sum b_a x^a \in \inn_w(I^h)$, and thus
$\inn_{\mathbf{v}_i}(\sum b_a x^a ) \in
\inn_{\mathbf{v}_i}(\inn_{w}(I^h))$, which would contradict $x^a \not
\in \inn_{\mathbf{v}_i}(\inn_{w}(I^h))$ unless $\sum b_a x^a = 0$.
Thus $\inn_{w}(g^{(i)})=x^{u_i}$.

Now we define a tropical basis $\mathcal{T}$ by taking any finite
generating set of $I$ and augmenting it by the polynomials
$g_i=g^{(i)}|_{x_0=1}$, where $g^{(i)}$ is as constructed above. Then
$\mathcal{T}$ is a generating set of $I$.  The intersection
$\bigcap_{f \in \mathcal T} \trop(V(f))$ contains $\trop(V(I))$ by the
definition of $\trop(V(I))$.  Consider an arbitrary weight vector $w
\in \Gamma^n \setminus \trop(V(I))$.  There exists an index $i$ such
that $\inn_{(0,w)}(I^h) = \inn_{w^{(i)}}(I^h)$, and this initial ideal
must contain a monomial since $w \not \in \trop(V(I))$.  The above
argument then shows that $\inn_{(0,w)}(g^{(i)})=x^{u_i}$, so $w \not \in
\trop(V(g_i))$.  Thus $w \not \in \bigcap_{f \in \mathcal T}
\trop(V(f))$ and so $\mathcal T$ is a finite tropical basis as
required.
\end{proof}

\begin{remark}
Hept and Theobald show in \cite{HeptTheobald} that if $X \subseteq
T^n$ is an irreducible $d$-dimensional variety, then there always
exist $f_0,\dots,f_{n-d} \in I(X)$ with $\trop(X) =
\bigcap_{i=0}^{n-d} \trop(V(f_i))$.  This means that if we drop the
ideal generation requirement then a tropical basis with $n-d+1$ elements
always exists.   Note, however, that the degrees of
the $f_i$ may be very large.  There are classical complete
intersections that are not the intersection of the tropicalizations
of any generating set of cardinality the codimension.

Alessandrini and Nesci give in \cite{AlessandriniNesci} a uniform
bound on the degrees of polynomials $f_i$ in a tropical basis for an
ideal $I$ that depends only on the Hilbert polynomial of a
homogenization of $I$.  Thus we can bound either the size, or the
degrees, of elements of a tropical basis.  However at the time of
writing a truly effective and efficient algorithm to compute tropical
bases does not exist.
\end{remark}

\begin{remark}  \label{r:noncanonical}
We warn that the polyhedral complex structure constructed here on
$\trop(X)$ is not canonical, but depends on the choice of
embedding of $T^n$ into $\mathbb P^n$ (or, algebraically, on the
choice of coordinates for the Laurent polynomial ring).  As an
explicit example of this phenomenon, let
$I=\langle a+b+c+d+e,3b+5c+7d+11e \rangle \subseteq \mathbb C[a^{\pm
    1},b^{\pm 1},c^{\pm 1}, d^{\pm 1}, e^{\pm 1}]$, and consider the
plane $X=V(I) \subseteq (\mathbb C^*)^5$.  The Gr\"obner fan of
$I$ has a one-dimensional
lineality space, spanned by $(1,1,1,1,1)$.  Modulo the lineality
space, the Gr\"obner fan structure on the tropical variety of $X$ has five
rays, and ten two-dimensional cones, which are the span any two of the
rays.  Let $\phi^* : \mathbb C[a^{\pm 1},b^{\pm 1},c^{\pm 1}, d^{\pm
    1},e^{\pm 1}] \rightarrow \mathbb C[a^{\pm 1},b^{\pm 1},c^{\pm 1},
  d^{\pm 1},e^{\pm 1}]$ be the automorphism given by $\phi^*(a)=ab$,
$\phi^*(b)=bc$, $\phi^*(c)=cd$, $\phi^*(d)=de$ and $\phi^*(e)=e$, and
let $\phi : (\mathbb C^*)^5 \rightarrow (\mathbb C^*)^5$ be the corresponding
morphism.  Let $Y = \phi(X) = V( {\phi^*}^{-1}(I))$.  The set $\trop(Y)$
is the image of $\trop(X)$ under the change of coordinates given by
$\trop(\phi^{-1})$, but the Gr\"obner fan structure on $\trop(Y)$ has
seven rays and twelve cones, as two of the two-dimensional cones are subdivided.  This can be verified using the software {\tt gfan}~\cite{gfan}.

A possible objection to this example is that the polyhedral structure
on $\trop(Y)$ refines the polyhedral structure on $\trop(X)$, so that
there is a more natural polyhedral structure.  However such a coarsest
polyhedral structure does not always exist; see \cite[Example
  5.2]{SturmfelsTevelev}.
\end{remark}

\begin{remark} \label{r:splitting}
Our construction of initial ideals depends on the choice of a
splitting $w \mapsto t^w$ of the valuation map $\val : K^* \rightarrow
\mathbb R$.  This is necessary to be able to compare initial ideals
with respect to different choices of $w$, as this choice makes our
initial ideals into ideals in $\K[x_1^{\pm 1},\dots,x_n^{\pm 1}]$ or
$\K[x_0,\dots,x_n]$.  

The more invariant choice recognizes that the Laurent polynomial ring
is the group ring $K[M]$, where $M \cong \mathbb Z^n$ is a lattice
with dual lattice $N = \Hom(M,\mathbb Z)$, and $\val(X)$ more
naturally lives in $N \otimes \mathbb R$, since $T^n \cong N \otimes
K^*$.  We then consider the tilted group ring $R[M]^w = \{ f = \sum
c_u x^u : \val(c_u)+w \cdot u \geq 0 \}$, which contains the ideal
$\mathfrak{m} = \{ f = \sum c_u x^u \in R[M]^w : \val(c_u)+w \cdot u
>0 \}$.  We can then define $\inn_w(I) = (I \cap R[M]^w)+\mathfrak{m}
\in R[M]^w/\mathfrak{m}$.  See \cite{PayneFiber} for this approach.

We note, though, that the choice of splitting is not a very serious
one.  Suppose $\phi_1, \phi_2 : \Gamma \rightarrow K^*$ are two
different splittings of $\val$, so $\val \circ \phi_1 = \val \circ
\phi_2 = \mathrm{id} : \Gamma \rightarrow \Gamma$.  These
homomorphisms induce isomorphisms $\phi_j : K[M] \rightarrow K[M]$ by
$x_i \mapsto \phi_j(w_i) x_i$ for $j=1,2$, which restrict to
isomorphisms $\phi_j : R[M]^w \rightarrow R[M]$ as we have
$\phi_j(\sum c_u x^u)=\sum c_u \phi_j(w \cdot u) x^u$.  Then if 
$\val(c_u)+w \cdot u \geq 0$, we have $\val(c_u \phi_j(w \cdot u) )
\geq 0$.  Thus $\psi = \phi_1 \circ \phi_2^{-1} : R[M] \rightarrow
R[M]$ is an automorphism.  Since $\psi$ is the restriction of the
automorphism of $K[M]$ given by $x_i \mapsto \phi_1(w_i)/\phi_2(w_i)
x_i$, $\psi$ maps the ideal $\mathfrak{m}$ to itself, so induces an
automorphism $\overline{\psi}: \K[M] \rightarrow \K[M]$.

This means that the two initial ideals of $I$ with respect to $w$
obtained using the splittings $\phi_1$ and $\phi_2$ are related by the
automorphism $\overline{\psi}$, so all invariants of the initial
ideal, such as dimension, are independent of the choice of splitting.
We also emphasize that such a choice is necessary to do computations.
One can view (parts of) tropical geometry as the computational arm of
rigid analytic geometry and Berkovich theory, so it is important not
to ignore the computational aspects.
\end{remark}

\section{Proofs}

This section contains the technical details needed to prove
Theorem~\ref{t:Grobnercomplex}.

\begin{lemma} \label{l:initial1}
For all $f \in K[x_0,\ldots,x_n]$ there exists $\epsilon>0$ such that
$\,\inn_v(\inn_w(f)) = \inn_{w+\epsilon' v}(f)\,$ for all $\epsilon'
\in \Gamma$ with $0 < \epsilon'<\epsilon$.
\end{lemma}

\begin{proof}

Let $f=\sum_{u \in {\mathbb N}^{n+1}} c_u x^u$.  Then
$\inn_w(f) \,\,= \sum_{u \in \mathbb N^{n+1} } \overline{c_u t^{w \cdot u -
W}} x^u$,
where $W = 
\trop(f)(w)$.  Let $W' = \min( v \cdot u : \val(c_u)+w \cdot u = W )$.
Then
$$ \inn_v(\inn_w(f)) \,\,= \sum_{v \cdot u = W'} \overline{c_u t^{w \cdot
u - W}} x^u.$$ 
For all sufficiently small $\epsilon > 0$, we have
$\, W+\epsilon W' = \trop(f)(w+\epsilon v) \,$ and 
 $$\{ u : \val(c_u) + (w + \epsilon v) \cdot u
=  W+\epsilon W' \} = \{u : \val(c_u) + w \cdot u = W, v \cdot u = W' \}.$$ 
This implies that $\inn_{w+\epsilon' v}(f) = \inn_v(\inn_w(f))\,$
for all $\epsilon' \in \Gamma$ with $0 < \epsilon'<\epsilon$.
\end{proof}

\begin{comment}

\begin{proof}
Conversely, the ideal $\inn_{w+\epsilon' v}(I)$ is finitely
generated by $h_1,\dots,h_r \in \K[x_0,\dots,x_n]$, with $h_i =
\inn_{w + \epsilon' v}(f'_i)$ for some $f'_i \in I$.
Then $ \inn_{w+\epsilon'v}(f'_i) = \inn_v(\inn_w(f'_i))  \in
\inn_v(\inn_w(I))$, and we conclude
$\inn_{w+\epsilon'v}(I) \subseteq \inn_v(\inn_w(I))$.

We now prove the lemma for a single polynomial $f =
\sum_{u \in {\mathbb N}^{n+1}} c_u x^u$.
Then 
$$\inn_w(f) \,\,= \sum_{u \in \mathbb N^{n+1} } \overline{c_u t^{w \cdot u -
W}} x^u,$$ 
where $W = 
\trop(f)(w)$.  Let $W' = \min( v \cdot u : \val(c_u)+w \cdot u = W )$.
Then
$$ \inn_v(\inn_w(f)) \,\,= \sum_{v \cdot u = W'} \overline{c_u t^{w \cdot
u - W}} x^u.$$ 
For all sufficiently small $\epsilon > 0$, we have
$\, W+\epsilon W' = \trop(f)(w+\epsilon v) \,$ and 
 $$\{ u : \val(c_u) + (w + \epsilon' v) \cdot u
=  W+\epsilon W' \} = \{u : \val(c_u) + w \cdot u = W, v \cdot u = W' \}.$$ 
This implies $\,\inn_{w+\epsilon' v}(f) = \inn_v(\inn_w(f))\,$
for all $\epsilon' \in \Gamma$ with $0 < \epsilon'<\epsilon$.
\end{proof}

\end{comment}

In standard Gr\"obner basis theory most attention is paid to initial ideals
that have a monomial generating set.  Such monomial ideals are useful
because their properties only depend on the set of monomials in the
ideal.  For example, a polynomial $f = \sum c_u x^u$ lies in a
monomial ideal if and only if every $x^u$ with $c_u \neq 0$ lies in
the ideal.  We next check that in this modified Gr\"obner theory
monomial initial ideals still exist.

\begin{lemma}  \label{l:monomial1}
Let $I$ be a homogeneous ideal in $K[x_0,\dots,x_n]$, and fix $w \in
\Gamma^{n+1}$.  Then there is $\mathbf{v} \in \mathbb Q^{n+1}$ and
$\epsilon >0$ for which both $\inn_{\mathbf{v}}(\inn_w(I))$ and
$\inn_{w+\epsilon \mathbf{v}}(I)$ are monomial ideals, and
$\inn_{\mathbf{v}}(\inn_w(I)) \subseteq \inn_{w+\epsilon
  \mathbf{v}}(I)$.
\end{lemma}

Note that in Corollary~\ref{l:initialofinitial} we will show that for sufficiently small $\epsilon>0$ these two initial ideals are equal.

\begin{proof}
Given any $\mathbf{v} \in \mathbb Q^{n+1}$, let $M_{\mathbf{v}}$
denote the ideal generated by all monomials in
$\inn_{\mathbf{v}}(\inn_w(I))$, and let $M^{\epsilon}_{\mathbf{v}}$
denote the ideal generated by all monomials in $\inn_{w+\epsilon
  \mathbf{v}}(I)$ for some $\epsilon>0$.  Choose $\mathbf{v} \in
\mathbb Q^{n+1}$ for which $M_{\mathbf{v}}$ is maximal with respect to
inclusion, so there is no $\mathbf{v}' \in \mathbb Q^{n+1}$ with
$M_{\mathbf{v}} \subsetneq M_{\mathbf{v}'}$.  This is possible since
the polynomial ring is Noetherian.  If $\inn_{\mathbf{v}}(\inn_w(I))$
is not a monomial ideal, then there is $f \in I$ with none of the
terms of $\inn_{\mathbf{v}}(\inn_w(f))$ lying in $M_{\mathbf{v}}$.
Choose $\mathbf{v}' \in \mathbb Q^{n+1}$ with $\inn_{\mathbf{v}'}
(\inn_{\mathbf{v}}(\inn_w(f)))$ a monomial; any $\mathbf{v}'$ for
which the face of the Newton polytope of
$\inn_{\mathbf{v}}(\inn_w(f))$ is a vertex suffices.  By
Lemma~\ref{l:initial1} there is $\epsilon'>0$ for which
$\inn_{\mathbf{v}+\epsilon' \mathbf{v}'}(\inn_w(f))$ is this monomial.
By choosing $\epsilon'$ sufficiently small we can guarantee that
$\inn_{\mathbf{v}+\epsilon' \mathbf{v}'}(\inn_w(I))$ contains all
generators of $M_{\mathbf{v}}$, as any generator $x^u$ is
$\inn_{\mathbf{v}}(\inn_w(f))$ for some $f \in I$ so this follows from
applying Lemma~\ref{l:initial1} to $\inn_{w}(f)$.  This contradicts
the choice of $\mathbf{v}$, so we conclude that
$\inn_{\mathbf{v}}(\inn_w(I))= M_{\mathbf{v}}$.

  Choose $f_1,\dots,f_s$ for which $\inn_{\mathbf{v}}(\inn_{w}(f_i)) =
  x^{u_i}$, where the $x^{u_i}$ generate $M_{\mathbf{v}}$.  By
  Lemma~\ref{l:initial1} there is $\epsilon >0$ for which $\inn_{w +
    \epsilon \mathbf{v}}(f_i)=x^{u_i}$ for all $i$, so for this
  $\epsilon$ we have $\inn_{\mathbf{v}}(\inn_w(I)) \subseteq \inn_{w +
    \epsilon \mathbf{v}}(I)$.  Suppose that $\mathbf{v}$ has been
  chosen from those $\mathbf{v}'$ with $\inn_{\mathbf{v}'}(\inn_w(I))$
  monomial so $M_{\mathbf{v}}^{\epsilon}$ is maximal with respect to
  inclusion.  Again, if $\inn_{w + \epsilon \mathbf{v}}(I)$ is not
  monomial then there is $f \in I$ with no term of $\inn_{w + \epsilon
    \mathbf{v}}(f)$ in $M_{\mathbf{v}}^{\epsilon}$, and we can choose
  $\mathbf{v}'$ and $\epsilon'$ as above so that
  $M_{\mathbf{v}}^{\epsilon} \subsetneq M_{\mathbf{v}+\epsilon'
    \mathbf{v}'}^{\epsilon}$ and
  $M_{\mathbf{v}}=M_{\mathbf{v}+\epsilon' \mathbf{v}'}$.  From this
  contradiction we conclude that $\inn_{w+\epsilon \mathbf{v}}(I)$ is
  also monomial, so we have constructed the desired $\mathbf{v} \in
  \mathbb Q^{n+1}$.
\end{proof}

We denote by $S_K$ the polynomial ring $K[x_0,\dots,x_n]$, and by
$S_{\K}$ the polynomial ring $\K[x_0,\dots,x_n]$.

\begin{lemma} \label{l:Hilbertfunction}
Let $I \subseteq K[x_0,\dots,x_n]$ be a homogeneous ideal, and let $w
\in \Gamma^{n+1}$ be such that $\inn_w(I)_d$ is the span of $\{ x^u : x^u \in \inn_w(I)_d \}$.   Then
the monomials not in $\inn_{w}(I)$ of degree $d$ form a $K$-basis for
$(S/I)_d$.  This implies that for arbitrary $w \in \Gamma^{n+1}$
 the Hilbert function of $I$ and $\inn_w(I)$ agree:
$$\dim_{K}(S_K/I)_d = \dim_{\K}(S_{\K}/\inn_w(I))_d \text{ for all degrees
}d.$$
\end{lemma}

\begin{proof}
Suppose first that $\inn_{w}(I)_d$ is the span of $\{ x^u : x^u \in
\inn_w(I)_d \}$.  Let $\mathcal B_d$ be the set of monomials of degree
$d$ not contained in $\inn_w(I)$.  We first show that, regarded as
elements of $(S/I)_d$, the set $\mathcal B_d$ is linearly independent.
Indeed, if this set were linearly dependent there would exist $f =
\sum c_u x^u \in I_d$, with $x^u \not \in \inn_w(I)$ whenever $c_u \neq
0$.  But then $\inn_w(f) \in \inn_w(I)_d$, which would mean that every
term of $\inn_w(f)$ is in $\inn_w(I)_d$, contradicting the construction
of $f$.  Since $|\mathcal B_d| = { n + d \choose n} - \dim_{\K}
\inn_w(I)_d$, this linear independence implies that $\dim_{\K}
\inn_w(I)_d \geq \dim_{K} I_d$.

For all monomials $x^u \in \inn_w(I)_d$, choose polynomials $f_u \in
I_d$ with $\inn_w(f_u)=x^u$.  We next note that the collection $\{ f_u
: x^u \in \inn_w(I)_d \}$ is linearly independent.  If not, there
would exist $a_u \in K$ not all zero with $\sum a_u f_u =0$.  Write
$f_u = x^u+ \sum c_{uv} x^v$.  Let $u'$ minimize $\val(a_u)+w \cdot u$
for all $ u \in \mathbb N^{n+1}$ with $x^u \in \inn_w(I)_d$.  Then
$a_{u'}+\sum_{u \neq u'} a_u c_{uu'}=0$, so there is $u'' \neq u'$
with $\val(a_{u''})+\val(c_{u''u'}) \leq \val(a_{u'})$.  But then
$\val(a_{u''})+\val(c_{u''u'})+w \cdot u' \leq \val(a_{u'})+w \cdot u'
\leq \val(a_{u''})+ w\cdot u''$, which contradicts
$\inn_w(f_{u''})=x^{u''}$.  This shows $\dim_{K} I_d \geq \dim_{\K}
\inn_w(I)_d$.  Thus, when $\inn_w(I)$ is a monomial ideal we have
$\dim_{K} (S_K/I)_d = \dim_{\K} (S_{\K}/\inn_w(I))_d$, and $\mathcal
B_d$ is a $K$-basis for $(S_K/I)_d$.

If $\inn_w(I)_d$ is not spanned by the monomials it contains, by
Lemma~\ref{l:monomial1} there is $\mathbf{v} \in \mathbb Q^{n+1}$ and
$\epsilon>0$ for which both $\inn_{\mathbf{v}}(\inn_w(I))_d$ and
$\inn_{w + \epsilon \mathbf{v}}(I)_d$ are spanned by the monomials
they contain and $\inn_{\mathbf{v}}(\inn_w(I))_d \subseteq
\inn_{w+\epsilon \mathbf{v}}(I)_d$.  By the previous calculation the
monomials not in $\inn_{w + \epsilon \mathbf{v}}(I)_d$ span $(S/I)_d$,
so if $x^u \in \inn_{w +\epsilon \mathbf{v}}(I)_d \setminus
\inn_{\mathbf{v}}(\inn_w(I))_d$ there is $f_u \in I_d$ of the form
$f_u = x^u + \sum c_v x^v$, where $c_v \neq 0$ implies that $x^v \not
\in \inn_{w + \epsilon \mathbf{v}}(I)_d$.  But then $\inn_w(f_u)$ is
supported on monomials not in $\inn_{\mathbf{v}}(\inn_w(I))_d$, so
$\inn_{\mathbf{v}}(\inn_w(f_u)) \not \in
\inn_{\mathbf{v}}(\inn_w(I))_d$.  From this contradiction we conclude
that $\inn_{w +\epsilon \mathbf{v}}(I)_d =
\inn_{\mathbf{v}}(\inn_w(I))_d$.

Standard Gr\"obner arguments imply $\dim_{\K}
(S_{\K}/\inn_w(I))_d = \dim_{\K}
(S_{\K}/\inn_{\mathbf{v}}(\inn_w(I)))_d$, and by the previous
calculations we have $\dim_K (S_K/I)_d =
\dim_{\K}(S_{\K}/\inn_{w+\epsilon \mathbf{v}}(I))_d$, so we conclude
that for any $w \in \Gamma^{n+1}$ we have $\dim_{K}(S_K/I)_d =
\dim_{\K}(S_{\K}/\inn_w(I))_d$ for all degrees $d$.

\end{proof}

\begin{corollary} \label{l:initialofinitial}
Let $I$ be a homogeneous ideal in $K[x_0,\dots,x_n]$, and let $w,
\mathbf{v} \in \Gamma^{n+1}$.  Then there is $\epsilon>0$ such that
for all $0 < \epsilon' <\epsilon$ with $\epsilon' \in \Gamma^{n+1}$ we
have $$\inn_{\mathbf{v}}(\inn_w(I)) = \inn_{w + \epsilon   \mathbf{v}}(I).$$
\end{corollary}

\begin{proof}
Let $\{g_1,\ldots,g_s\} \subset \K[x_0, \ldots, x_n]$ be a generating
set for $\inn_{\mathbf{v}}(\inn_w(I))$, with each generator $g_i$ of
the form $\inn_{\mathbf{v}}(\inn_w(f_i))$ for some $f_i \in I$. We
choose $\epsilon$ to be the minimum of the $\epsilon_i$ from
Lemma~\ref{l:initial1}.  Then $g_i=\inn_{\mathbf{v}}(\inn_w(f_i)) =
\inn_{w+\epsilon' \mathbf{v}}(f_i)$ for any $\epsilon'<\epsilon$, so
$\inn_{\mathbf{v}}(\inn_w(I)) \subseteq \inn_{w+\epsilon'
  \mathbf{v}}(I)$.  But by Lemma~\ref{l:Hilbertfunction} both
$\inn_{\mathbf{v}}(\inn_w(I))$ and $\inn_{w + \epsilon'
  \mathbf{v}}(I)$ have the same Hilbert function as $I$, so this
containment cannot be proper.
\end{proof}

\begin{proposition} \label{p:finite}
Let $I$ be a homogeneous ideal in $K[x_0,\dots,x_n]$.  There are only
a finite number of different monomial  initial ideals $\inn_w(I)$ as $w$ varies over
$\Gamma^{n+1}$.
\end{proposition}

\begin{proof}
  If this were not the case, by \cite[Theorem 1.1]{MaclaganAntichains}
there would be $w_1,w_2 \in \Gamma^{n+1}$ with $\inn_{w_2}(I)
\subsetneq \inn_{w_1}(I)$, where both initial ideals are monomial
ideals.  Fix $x^u \in \inn_{w_1}(I) \setminus \inn_{w_2}(I)$.  By
Lemma~\ref{l:Hilbertfunction} the monomials of degree $\deg(x^u)$ not
in $\inn_{w_1}(I)$ form a $K$-basis for $S/I$, so there is $f_u \in I$
with $f_u = x^u + \sum c_v x^v$ where whenever $c_v \neq 0$ we have
$x^v \not \in \inn_{w_1}(I)$.  But then $\inn_{w_2}(f_u) \in
\inn_{w_2}(I)$, and since $\inn_{w_2}(f_u)$ is a monomial ideal this
means that all of its terms lie in $\inn_{w_2}(I)$.  However all
monomials appearing in $\inn_{w_2}(f_u)$ appear in $f_u$, so this is a
contradiction, and thus there are only a finite number of monomial
initial ideals of $I$.
\end{proof}

Fix a homogeneous ideal $I \subseteq K[x_0,\dots,x_n]$.
Proposition~\ref{p:finite} guarantees that there are only finitely
many different monomial initial ideals of $I$.  Let $D$ be the maximum degree
of a minimal generator of any monomial initial ideal of $I$.

For any fixed degree $d$ let $s = \dim_{K}(I_d)$.  Choose a basis
$f_1,\dots,f_s$ for $I_d$, and let $A_d$ be the corresponding $s \times
{n+d \choose n}$ matrix recording the coefficients of the polynomials
$f_i$.  This matrix has columns indexed by the monomials $\mathcal
M_d$ in $K[x_0,\dots,x_n]$ of degree $d$, so $(A_d)_{iu}$ is
the coefficient of $x^{u}$ in $f_i$.  Note that the maximal
minors of this matrix are independent of the choice $f_1,\dots,f_s$ of
basis, as they are the Pl\"ucker coordinates of the element $I_d$ in
the Grassmannian $\Gr(s, S_d)$.  For $J \subseteq \mathcal M_d$ with
$|J| = s$, we denote by $A_d^J$ the $s \times s$ minor of $A_d$
indexed by columns labeled by those monomials in $J$.

Let $g_d \in K[x_0,\dots,x_n]$ be given by
$$g_d = \sum_{I \subseteq \mathcal M_d, |I|=s} \det(A_d^I)
\prod_{\mathbf{u} \in I} x^u.$$ Let $g = \prod_{d=1}^D g_d$.  The
function $\trop(g) : \mathbb R^{n+1} \rightarrow \mathbb R$ is
piecewise-linear.  Let $\Sigma_{\trop(g)}$ be the coarsest polyhedral
complex for which $\trop(g)$ is linear on each polyhedron in
$\Sigma_{\trop(g)}$.  Note that $\Sigma_{\trop(g)}$ is a
$\Gamma$-rational polyhedral complex.

\begin{theorem} \label{t:tropicalGrobnerComplex}
Fix a homogeneous ideal $I \subseteq K[x_0,\dots,x_n]$, and let $g_d,
g$ and $\Sigma_{\trop(g)}$ be as above.  Fix $w \in \Gamma^{n+1}$ in the interior of a maximal polyhedron 
$\sigma \in \Sigma_{\trop(g)}$.   Then
$\sigma = \overline{C_{I}[w]}$.
\end{theorem}

\begin{proof}
We show that if $w' \in \Gamma^{n+1}$ lies in the interior of $\sigma$
if and only if $\inn_{w'}(I) = \inn_w(I)$.  Note that
$\Sigma_{\trop(g)}$ is the common refinement of the polyhedral
complexes $\Sigma_{\trop(g_d)}$ for $d \leq D$, where
$\Sigma_{\trop(g_d)}$ is the coarsest polyhedral complex for which
$\trop(g_d)$ is linear on each polyhedron.  Thus it suffices to
restrict to a fixed $d \leq D$, and let $\sigma_d$ be the polyhedron
of $\Sigma_{\trop(g_d)}$ containing $\sigma$.  In what follows we show
that $w' \in \Gamma^{n+1}$ lies in the interior of $\sigma_d$ if and
only if $\inn_{w'}(I)_d = \inn_w(I)_d$.  This suffices because
$\inn_{w'}(I)=\inn_{w}(I)$ if and only if $\inn_{w'}(I)_d =
\inn_{w}(I)_d$ for all $d \leq D$.

For the only if direction, note that if $w'$ lies in the interior of
$\sigma_d$ then the minimum in $\trop(g_d)$ is achieved at the same
term for $w$ and for $w'$.  Since $\sigma_d$ is a maximal polyhedron,
this minimum is achieved at only one term, which we may assume is the
one indexed by $J \in \mathcal M_d$.

Let $\tilde{A}$ be the $s \times s$ submatrix of $A_d$ containing
those columns corresponding to monomials in $J$, and consider the
matrix $A'=\tilde{A}^{-1}A_d$.  This shifts the valuations of the
minors: $\val(A'^{J'})=\val(A^{J'}_d)-\val(\det(\tilde{A}))$.  The
matrix $A'$ has an identity matrix in the columns indexed by $J$, so
each row gives a polynomial in $S_d$ indexed by $x^u \in J'$.  Let
$\tilde{f}_u = x^u + \sum_{x^v \not \in J'} c_v x^v$ be the polynomial
indexed by $x^u$.  Note that the minor of $A'$ indexed by $J_v=J
\setminus \{ x^u\} \cup \{x^v \}$ for $x^v \not \in J$ is $c_v$, up to
sign, so
\begin{align*} \val(A'^{J_v})+\sum_{x^{u'} \in
  J_v} w \cdot u' &=  \val(A_d^{J_v}) - \val(\det(\tilde{A})) + \sum_{x^{u'} \in
  J_v} w \cdot u' \\
& > \val(A_d^{J}) - \val(\det(\tilde{A}))+\sum_{x^{u'} \in J} w \cdot u'\\
&= \val({A'}_d^{J}) + \sum_{x^{u'} \in J_v} w \cdot u'+ w \cdot u - w\cdot v\\
&= 0+
\sum_{x^{u'} \in J_v} w \cdot u'+ w \cdot u - w\cdot v\\
\end{align*}
Thus $\val(c_v)+ w \cdot v > w \cdot u$ for any $v$ with $x^v \not \in
J$, so $\inn_w(\tilde{f}_u)=x^u$.  This means that $x^u \in
\inn_w(I)_d$.  Since $\dim_{\K} \inn_{w}(I)_d=s$ by
Lemma~\ref{l:Hilbertfunction}, $J$ is precisely the collection of
monomials in $\inn_w(I)_d$.  Since $|J|=s=\dim_{k} \inn_{w}(I)_d =
\inn_{w'}(I)_d$ we have $\inn_w(I)_d = \inn_{w'}(I)_d$ as required.
Note that this also shows that $\inn_w(I)$ is a monomial ideal, since
in all degrees $d$ up to the bound $D$ on its generators $\inn_w(I)_d$
is spanned by monomials in $\inn_w(I)$

For the if direction, suppose that $w'$ does not lie in the interior
of $\sigma_d$.  This means that there is some $J' \in \mathcal M_d$
with $J' \neq J$ and $\val(A_d^{J'})+\sum_{u \in J'} w' \cdot u \leq
\val(A_d^L)+\sum_{u \in L} w' \cdot u$ for any $L$.  We may choose
$J'$ so that $\sum_{u \in J'} u$ is a vertex of the convex hull of all
$\sum_{u \in J''} u$ with $J''$ satisfying the inequality.  This means
that there is $\mathbf{v} \in \mathbb Q^{n+1}$ with $\mathbf{v} \cdot
(\sum_{u \in J'} u - \sum_{u \in J''} u) <0$ for any such $J''$.  Then
for sufficiently small $\epsilon >0$ we have $\val(A_d^{J'})+\sum_{u
  \in J'} (w'+\epsilon \mathbf{v}) \cdot u < \val(A_d^L)+\sum_{u \in
  L} (w'+\epsilon \mathbf{v}) \cdot u$ for all $L \in \mathcal M_d$
with $L \neq J'$.  The above argument then shows that
$\inn_{w'+\epsilon \mathbf{v}}(I)_d = \spann \{ x^u : x^u \in J'\}$.
By Corollary~\ref{l:initialofinitial} we have $\inn_{w'+\epsilon
  \mathbf{v}}(I) = \inn_{\mathbf{v}}(\inn_{w'}(I))$, so this means
that $\inn_{w'}(I)_d$ is not the span of those monomials in $J$, and
thus $\inn_{w'}(I)_d \neq \inn_w(I)_d$.
\end{proof}

Theorem~\ref{t:Grobnercomplex} is now a straightforward corollary of Theorem~\ref{t:tropicalGrobnerComplex}.

\begin{proof}[Proof of Theorem~\ref{t:Grobnercomplex}]
Theorem~\ref{t:tropicalGrobnerComplex} states that all top-dimensional
regions of the $\Gamma$-rational polyhedral complex
$\Sigma_{\trop(f)}$ are of the form $\overline{C_{I}[w]}$ for some $w
\in \Gamma^{n+1}$ with $\inn_w(I)$ a monomial ideal.  For any $w \in
\Gamma^{n+1}$ with $\inn_w(I)$ a monomial ideal by
Corollary~\ref{l:initialofinitial} we have $\inn_{w + \epsilon
  \mathbf{v}}(I) = \inn_w(I)$ for all $\mathbf{v} \in \mathbb Q^{n+1}$ and
all sufficiently small $\epsilon$.  This means that such a
$\overline{C_I[w]}$ is full-dimensional, so it must be one of the
top-dimensional regions of $\Sigma_{\trop(f)}$, as for $w \neq w'$ the
regions $C_I[w]$ and $C_I[w']$ are either disjoint or coincide.  It
thus remains to show that if $\inn_w(I)$ is not a monomial ideal, then
$\overline{C_I[w]}$ is a face of some $\overline{C_I[w']}$ with
$\inn_{w'}(I)$ a monomial ideal.

This follows from Corollary~\ref{l:initialofinitial} and
 Lemma~\ref{l:monomial1}.  Indeed, by Lemma~\ref{l:monomial1}
there is some $\mathbf{v} \in \mathbb Q^{n+1}$ with
$\inn_{\mathbf{v}}(\inn_w(I))$ a monomial ideal, and by
Corollary~\ref{l:initialofinitial} there is $\epsilon>0$ for which
$\inn_{w+\epsilon \mathbf{v}}(I)=\inn_{\mathbf{v}}(\inn_w(I))$.  Let
$w' = w +\epsilon \mathbf{v}$.  Let $g_1,\dots,g_s$ be a Gr\"obner
basis for $I$ with respect to $w'$, so $\inn_{w'}(I) = \langle
\inn_{w'}(g_1),\dots,\inn_{w'}(g_s) \rangle$.  Write $g_i = x^{u_i} +
\sum c_{iv} x^v$, where $\inn_w(g_i) = x^{u_i}$.  We may assume, as in
the proof of Lemma~\ref{l:Hilbertfunction}, that $c_{iv} \neq 0$
implies that $x^v \not \in \inn_{w'}(I)$.  Then the polyhedron
$\overline{C_I[w']}$ has the following inequality description:
$$\overline{C_I[w']} = \{ x \in \mathbb R^{n+1} : x \cdot u_i \leq
\val(c_{iv}) + x \cdot v : 1 \leq i \leq s \}.$$ To see this, first
note that for any $\widetilde{w} \in C_I[w']$ we have all inequalities
on the righthand side satisfied properly.  Otherwise there would some
monomial not in $\inn_{w'}(I)$ appearing in some
$\inn_{\widetilde{w}}(g_i)$, which would contradict this polynomial
lying in the monomial ideal $\inn_{\widetilde{w}}(I)=\inn_{w'}(I)$.
As the righthand set is closed, this shows the containment of
$\overline{C_I[w']}$ in the righthand set.  For the reverse inclusion,
if $\widetilde{w} \in \Gamma^{n+1}$ lies outside the righthand set,
there is some $g_i$ for which $\inn_{\widetilde{w}}(g_i)$ does not
contain $x^{u_i}$ in its support.  Let $b=\widetilde{w} \cdot u_i -
\min \{\val(c_{vi}+ \widetilde{w} \cdot v : c_{vi} \neq 0\}$.  By
assumption $b>0$.  If $\widetilde{w} \in \overline{C_I[w']}$ then for
all $\epsilon>0$ there is $\mathbf{u}'$ with $|\mathbf{u}'|<\epsilon$
and $\widetilde{w}+\mathbf{u}' \in C_I[w']$.  Choose $\epsilon >0$
sufficiently small so that all $\mathbf{u}'$ with
$|\mathbf{u}'|<\epsilon$ satisfy $\mathbf{u}' \cdot (v-u_i)<b/2$ for
all $v$ with $c_{vi} \neq 0$.  Then
$\inn_{\widetilde{w}+\mathbf{u}'}(g_i) \in
\inn_{\widetilde{w}+\mathbf{u}'}(I)$ does not contain $x^{u_i}$ in its
support.  Since $\inn_{\widetilde{w}+\mathbf{u}'}(I)=\inn_{w'}(I)$ is
a monomial ideal, all terms of $\inn_{\widetilde{w}+\mathbf{u}'}(g_i)$
must lie in $\inn_{w'}(I)$, which is a contradiction, so such
$\widetilde{w}$ do not lie in $\overline{C_I[w']}$, and thus
$\overline{C_I[w']}$ has the claimed description.

The argument in the second paragraph implies that $C_I[w]$ lies in
$\overline{C_I[w']}$, so we just need to show that it is a face.  Note
that $\{ \inn_w(g_1),\dots,\inn_w(g_s) \}$ is a Gr\"obner basis for
$\inn_w(I)$ with respect to $\mathbf{v}$.  If $\widetilde{w} \in
\Gamma^{n+1}$ satisfies $\inn_{\widetilde{w}}(I) = \inn_w(I)$, then we
must have $\inn_{\widetilde{w}}(g_i) = \inn_w(g_i)$.  If not,
$\inn_{\widetilde{w}}(g_i)$ must still have $x^{u_i}$ in its support,
or we would not have $\inn_{\mathbf{v}}(\inn_{\widetilde{w}}(I)$ equal
to the monomial ideal $\inn_{w'}(I)$.  But then
$\inn_{\widetilde{w}}(g_i)-\inn_w(g_i) \in \inn_w(I)$, and this
polynomial does not contain any monomials from $\inn_{w'}(I)$,
contradicting $\inn_{\mathbf{v}}(\inn_w(I)) = \inn_{w'}(I)$.  Thus
$\widetilde{w}$ lies in the polyhedron 
\begin{align*}
\{ x \in \mathbb R^{n+1} : & \,\, u_i \cdot x \leq \val(c_{iv}) + x \cdot
v \text{ and }  u_i \cdot x = v' \cdot x \text{ for all }  1 \leq i \leq s \\ & \text{ and } x^{v'} \text{ in the
  support of } \inn_w(g_i) \}.\\
\end{align*}
 On the other hand, any $\widetilde{w}
\in \Gamma^{n+1}$ lying in this set has $\inn_{\widetilde{w}}(g_i)=
\inn_w(I)$, so $\inn_w(I) \subseteq \inn_{\widetilde{w}}(I)$, and so
by Lemma~\ref{l:Hilbertfunction} we have equality, so $\widetilde{w}
\in C_I[w]$.  Since this polyhedron is the intersection of
$\overline{C_I[w']}$ with a supporting subspace it is a face as
required.
\end{proof}
\begin{remark}
The construction of the Gr\"obner complex as the regions where a piece-wise linear 
tropical function is linear  shows that this polyhedral complex is a {\em regular
  subdivision}.  This notion originates in the work of Gelfand,
Kapranov, and Zelevinsky \cite[Chapter 7]{GKZ}, where such
subdivisions were called coherent; see also \cite[Chapter
  5]{TriangulationsBook}.  The content here is that the piecewise
linear function $\trop(f)$ is concave.
\end{remark}

\begin{remark}
The polynomial $g$ is homogeneous of degree $L=\sum_{d=1}^D
\dim_{K}(I_d)$.  Write $g = \sum c_u x^u$, where the sum is over $u
\in \mathbb N^{n+1}$ with $|u|=L$.  When $K$ has the trivial
valuation, the regions where $\trop(g)$ is linear are the cones of the
normal fan of the polytope $\conv(u \in \mathbb N^{n+1} : c_u \neq
0)$.  This polytope is known as the state polytope of $I$, and was
first described in \cite{BayerMorrison}.  The construction given above
mimics this construction; see \cite[Chapter 2]{GBCP} for an exposition
in this case.  When $K$ has a nontrivial valuation, the Gr\"obner
complex agrees with the normal fan to the state polytope of $I$ for
large $w$, and is the dual complex to a regular subdivision of the
state polytope.
\end{remark}

\begin{remark}
When $K$ has the trivial valuation we do not need to assume that the
ideal $I$ is homogeneous to define the Gr\"obner fan.  In this case Anders
Jensen gave an example in \cite{AndersNonRegular} of an ideal $I
\subseteq \mathbb C[x_1,x_2,x_3,x_4]$ for which the Gr\"obner fan is
not a regular subdivision.  However if we take $X \subset T^4$ to be
the variety defined by the ideal $I \mathbb C[x_1^{\pm 1},x_2^{\pm 1},
  x_3^{\pm 1}, x_4^{\pm 4}]$, then $\trop(X)$ is the support of a
subcomplex of this Gr\"obner fan, and also the support of a
subcomplex of a regular subdivision.  This is not a contradiction, as
the regular subdivision coming from the Gr\"obner fan of the
homogenization can be much finer than the nonregular one.
\end{remark}

\noindent 
{\bf Acknowledgements.} 
This paper owes its existence to the Bellairs workshop, so I am
grateful to the organizers.   I
suspect few conferences in my career will beat it on both the interest
of the mathematics and the pleasantness of the surroundings.  I also
wish to acknowledge an intellectual debt to Bernd Sturmfels from whom
I learnt many of these ideas, and with whom I have discussed many of
the details of exposition over the writing of \cite{TropicalBook}.  Thanks also to Anders Jensen and Josephine Yu for help with Remark~\ref{r:noncanonical}, and to Frank Sottile for the impetus for Remark~\ref{r:splitting}.

\end{document}